\documentclass[10pt]{amsart}

\usepackage[colorlinks]{hyperref}
\usepackage{color,graphicx,shortvrb}
\usepackage[latin 1]{inputenc}

\usepackage[active]{srcltx} 

\usepackage{enumerate}
\usepackage{amssymb}
\usepackage{tikz-cd}

\usepackage [ all ]{xy}

\newtheorem{theorem}{Theorem}[section]

\newtheorem{corollary}[theorem]{Corollary}

\newtheorem{definition}[theorem]{Definition}
\newtheorem{lemma}[theorem]{Lemma}

\newtheorem{remark}[theorem]{Remark}

\def\J#1#2#3{ \left\{ #1,#2,#3 \right\} }
\def\RR{{\mathbb{R}}}
\def\NN{{\mathbb{N}}}
\def\11{\textbf{$1$}}
\def\11b#1{\mathbf{1}_{_{#1}}}

\DeclareGraphicsExtensions{.jpg,.pdf,.png,.eps}

\begin{document}

\title[Surjective isometries between unitary groups]{Can one identify two unital JB$^*$-algebras by the metric spaces determined by their sets of unitaries?}

\author[M. Cueto-Avellaneda, A.M. Peralta]{Mar{\'i}a Cueto-Avellaneda, Antonio M. Peralta}

\address[M. Cueto-Avellaneda, A.M. Peralta]{Departamento de An{\'a}lisis Matem{\'a}tico, Facultad de
Ciencias, Universidad de Granada, 18071 Granada, Spain.}
\email{mcueto@ugr.es, aperalta@ugr.es}




\date{}

\begin{abstract} Let $M$ and $N$ be two unital JB$^*$-algebras and let $\mathcal{U} (M)$ and $\mathcal{U} (N)$ denote the sets of all unitaries in $M$ and $N$, respectively. We prove that the following statements are equivalent:\begin{enumerate}[$(a)$]
\item $M$ and $N$ are isometrically isomorphic as (complex) Banach spaces;
\item $M$ and $N$ are isometrically isomorphic as real Banach spaces;
\item There exists a surjective isometry $\Delta: \mathcal{U}(M)\to \mathcal{U}(N).$
\end{enumerate}  We actually establish a more general statement asserting that, under some mild extra conditions, for each surjective isometry $\Delta:\mathcal{U} (M) \to \mathcal{U} (N)$ we can find a surjective real linear isometry $\Psi:M\to N$ which coincides with $\Delta$ on the subset $e^{i M_{sa}}$. If we assume that $M$ and $N$ are JBW$^*$-algebras, then every surjective isometry $\Delta:\mathcal{U} (M) \to \mathcal{U} (N)$ admits a (unique) extension to a surjective real linear isometry from $M$ onto $N$. This is an extension of the Hatori--Moln{\'a}r theorem to the setting of JB$^*$-algebras.
\end{abstract}

\maketitle
\thispagestyle{empty}

\section{Introduction}

Every surjective isometry between two real normed spaces $X$ and $Y$ is an affine mapping by the Mazur--Ulam theorem. It seems then natural to ask whether the existence of a surjective isometry between two proper subsets of $X$ and $Y$ can be employed to identify metrically both spaces. By a result of P. Mankiewicz (see \cite{Mank1972}) every surjective isometry between convex bodies in two arbitrary normed spaces can be uniquely extended to an affine function between the spaces. The so-called Tingley's problem, which ask if a surjective isometry between the unit spheres of two normed spaces can be also extended to a surjective linear isometry between the spaces, came out in the eighties (cf. \cite{Ting1987}). To the best of our knowledge, Tingley's problem remains open even for two dimensional spaces (see \cite{CabSan19} where it is solved for non-strictly convex two dimensional spaces). A full machinery has been developed in the different partial positive solutions to Tingley's problem in the case of classical Banach spaces, C$^*$- and operator algebras and JB$^*$-triples  (see, for example the references \cite{BeCuFerPe2018, CabSan19, ChenDong2011, CuePer18, CuePer19, FerGarPeVill17, FerJorPer2018, FerPe17c, FerPe17d, FerPe18Adv, JVMorPeRa2017, KalPe2019, Mori2017, MoriOza2020, Pe2019, PeTan19, WH19} and the surveys \cite{YangZhao2014,Pe2018}).\smallskip

The question at this stage is whether in Tingley's problem the unit spheres can be reduced to strictly smaller subsets. Even in the most favorable case of a finite dimensional normed space $X$, we cannot always conclude that every surjective isometry on the set of extreme points of the closed unit ball of $X$ can be extended to a surjective real linear isometry on $X$ (see \cite[Remark 3.15]{CuePer18}). So, the sets of extreme points is not enough to determine a surjective real linear isometry. The existence of an additional structure on $X$ provides new candidates, this is the case of unital C$^*$-algebras. In a unital C$^*$-algebra $A$, the set $\mathcal{U} (A)$ of all unitary elements in $A$ is, in general, strictly contained in the set of all extreme points of the closed unit ball of $A$. The symbol $A_{sa}$ will stand for the set of self-adjoint elements in $A$. We recall that an element $u$ in $A$ is called \emph{unitary} if $u u^* =\11b{A} = u^* u$, that is, $u$ is invertible with inverse $u^*$. The set of all unitaries in $A$ will be denoted by $\mathcal{U}(A)$. It is well known that $\mathcal{U}(A)$ is contained in the unit sphere of $A$ and it is a subgroup of $A$ which is also self-adjoint (i.e., $u^*$ and $u v$ lie in $\mathcal{U}(A)$ for all $u,v \in \mathcal{U}(A)$). However, the set $\mathcal{U}(A)$ is no longer stable under Jordan products of the form $a\circ b :=\frac12 (a b + ba)$. Namely, let $u,v\in \mathcal{U}(A)$ the element $w = u\circ v$ is a unitary if and only if $\11b{A} = w w^* = w^* w$, that is, $$ \11b{A}= \frac14 (u v + v u ) (v^* u^* + u^* v^* ) = \frac14 ( 2\cdot \11b{A} + u v u^* v^* + v u v^* u^* ),$$ equivalently, $\11b{A} = \frac{u v u^* v^* + v u v^* u^*}{2}$ and thus $u v u^* v^* = v u v^* u^* =\11b{A},$ because $\11b{A}$ is an extreme point of the closed unit ball of $A$. In particular $u v = v u$. That is $u\circ v\in \mathcal{U}(A)$ if and only if $u$ and $v$ commute. Despite the instability of unitaries under Jordan products, expressions of the form $u v u $ lie in $\mathcal{U}(A)$ for all $u,v\in \mathcal{U}(A),$ and they can be even expressed in terms of the Jordan product because $uv u = 2 (u\circ v) \circ u - u^2 \circ v$.\smallskip

O. Hatori and L. Moln{\'a}r proved in \cite[Theorem 1]{HatMol2014}, that for each surjective isometry $\Delta: \mathcal{U} (A) \to \mathcal{U} (B)$, where $A$ and $B$ are unital C$^*$-algebras, the identity $\Delta (e^{i A_{sa}}) = e^{i B_{sa}}$ holds, and there is a central projection $p \in B$ and a Jordan $^*$-isomorphism $J : A\to B$ satisfying $$\Delta (e^{ix}) = \Delta(1) (p J(e^{ix}) + (1-p) J(e^{ix})^*), \ \ (x\in A_{sa}).$$ In particular $A$ and $B$ are Jordan $^*$-isomorphic. Actually, every surjective isometry between the unitary groups of two von Neumann algebras admits an extension to a surjective real linear isometry between these algebras (see \cite[Corollary 3]{HatMol2014}). These influencing results have played an important role in some of the recent advances on Tingley's problem and in several other problems.\smallskip

Let us take a look at some historical precedents. S. Sakai proved in \cite{Sak55} that if $M$ and $N$ are AW$^*$-factors, $\mathcal{U}(M)$, $\mathcal{U}(N)$ their respective unitary groups, and $\rho$ a uniformly continuous group isomorphism from $\mathcal{U}(M)$ into $\mathcal{U}(N)$, then there is a unique map $f$ from $M$ onto $N$ which is either a linear or conjugate linear $^*$-isomorphism and which agrees with $\rho$ on $\mathcal{U}(M)$. In the case of W$^*$-factors not of type $I_{2n}$ the continuity assumption was shown to be superfluous by H.A. Dye in \cite[Theorem 2]{Dye55}. In the results by Hatori and Moln{\'a}r, the mapping $\Delta$ is merely a distance preserving bijection between the unitary groups of two unital C$^*$-algebras or two von Neumann algebras.\smallskip

The proofs of the Hatori--Moln{\'a}r theorems are based, among other things, on a study on isometries and maps compatible with inverted Jordan triple products on groups by O. Hatori, G. Hirasawa, T. Miura, L. Molnár \cite{HatHirMiuMol2012}. Despite of the attractive terminology, the study of the surjective isometries between the sets of unitaries of two unital JB$^*$-algebras has not been considered. There are different diffulties which are inherent to the Jordan setting. As we commented above, the set of unitary elements in a unital C$^*$-algebra is not stable under Jordan products.\smallskip

Motivated by the pioneering works of I. Kaplansky, JB$^*$-algebras were introduced as a Jordan generalization of C$^*$-algebras (see subsection \ref{subsection background} for the detailed definitions). For example every Jordan self-adjoint subalgebra of a C$^*$-algebra is a JB$^*$-algebra (these JB$^*$-algebras are called JC$^*$-algebras), but there exists exceptional JB$^*$-algebras which cannot be represented as JC$^*$-algebras.\smallskip

Unitaries in unital C$^*$-algebras and JB$^*$-algebras have been intensively studied. They constitute the central notion in the Russo--Dye theorem \cite{RuDye} and its JB$^*$-algebra-analogue in the Wright--Youngson--Russo--Dye theorem \cite{WriYou77}, which are milestone results in the field of functional analysis. The interest remains very active, for example, we recently obtained a metric characterization of unitary elements in a unital JB$^*$-algebra (cf. \cite{CuPe20}).\smallskip

By the Gelfand--Naimark theorem, every unitary $u$ in a unital C$^*$-algebra $A$ can be viewed as a unitary element in the algebra $B(H)$, of all bounded linear operators on a complex Hilbert space $H$, in such a way that $u$ itself is a unitary on $H$. Consequently, one-parameter unitary groups in $A$ are under the hypotheses of some well known results like Stone's one-parameter theorem. However, unitary elements in a unital JB$^*$-algebra $M$ cannot always be regarded as unitaries on some complex Hilbert space $H$. The lacking of a suitable Jordan version of Stone's one-parameter theorem for JB$^*$-algebras leads us to establish an appropriate result for uniformly continuous one-parameter groups of unitaries in an arbitrary unital JB$^*$-algebra in Theorem \ref{t Jordan unitary groups version of Stone's theorem}.\smallskip

Let $M$ and $N$ denote two arbitrary unital JB$^*$-algebras whose sets of unitaries are denoted by $\mathcal{U} (M)$ and $\mathcal{U} (N),$  respectively. In our first main result (Theorem \ref{t HM for unitary sets of unitary JB*-algebras}) we prove that for each surjective isometry $\Delta: \mathcal{U} (M)\to \mathcal{U} (N)$ satisfying one of the following statements:
\begin{enumerate}[$(1)$]\item $\|\11b{N}-\Delta(\11b{M})\|<2$;
\item there exists a unitary $\omega_0$ in $N$ such that $U_{\omega_0} (\Delta(\11bM)) = \11bN,$
\end{enumerate} there exists a unitary $\omega$ in $N$ satisfying $$\Delta(e^{i M_{sa}})=U_{\omega^*}(e^{i N_{sa}}),$$ Furthermore, we can find a central projection $p\in N$, and a Jordan $^*$-isomorphism $\Phi:M\to N$ such that $$\begin{aligned}\Delta(e^{ih}) &=  U_{\omega^*}\left( p \circ \Phi(e^{ih}) \right) + U_{\omega^*}\left( (\11b{N}-p) \circ\Phi( e^{ih})^*\right)\\
&=  P_2(U_{\omega^*}( p)) U_{\omega^*}(\Phi(e^{ih}))  + P_2(U_{\omega^*}(\11b{N} -p)) U_{\omega^*}(\Phi( (e^{ih})^*)),
\end{aligned}$$ for all $h\in M_{sa}$. Consequently, the restriction $\Delta|_{e^{i M_{sa}}}$ admits a (unique) extension to a surjective real linear isometry from $M$ onto $N$. We remark that $\Delta$ is merely a distance preserving bijection.\smallskip

Among the consequences of the previous result we prove that the following statements are equivalent for any two unital JB$^*$-algebras $M$ and $N$:
\begin{enumerate}[$(a)$]
\item $M$ and $N$ are isometrically isomorphic as (complex) Banach spaces;
\item $M$ and $N$ are isometrically isomorphic as real Banach spaces;
\item There exists a surjective isometry $\Delta: \mathcal{U}(M)\to \mathcal{U}(N)$
\end{enumerate} (see Corollary \ref{c two unital JB*algebras are isomorphic iff their unitaries are}).\smallskip

Finally, in Theorem \ref{t HM for unitary sets of unitary JBW*-algebras} we prove that any surjective isometry between the sets of unitaries of any two JBW$^*$-algebras admits a (unique) extension to a surjective real linear isometry between these algebras.\smallskip

Our proofs, which are completely independent from the results for C$^*$-algebras, undoubtedly benefit from results in JB$^*$-triple theory. This beautiful subject (meaning the theory of JB*- and JBW*- triples) makes simpler and more accessible our arguments.

\subsection{Definitions and background}\label{subsection background}

A complex (respectively,  real) \emph{Jordan algebra} $M$ is a (non-necessarily associative) algebra over the complex (respectively, real) field whose product is abelian and satisfies the so-called \emph{Jordan identity}: $(a \circ b)\circ a^2 = a\circ (b \circ a^2)$ ($a,b\in M$). A \emph{normed Jordan algebra} is a Jordan algebra $M$ equipped with a norm, $\|.\|$, satisfying $\|
a\circ b\| \leq \|a\| \ \|b\|$ ($a,b\in M$). A \emph{Jordan Banach algebra} is a normed Jordan algebra whose norm is complete.
Every real or complex associative Banach algebra is a real Jordan Banach algebra with respect to the product $a\circ b: = \frac12 (a b +ba)$. \smallskip

Let $M$ be a Jordan Banach algebra. Given $a,b\in M$, we shall write $U_{a,b}:M\to M$ for the bounded linear operator defined by $$U_{a,b} (x) =(a\circ x) \circ b + (b\circ x)\circ a - (a\circ b)\circ x,$$ for all $x\in M$. The mapping $U_{a,a}$ will be simply denoted by $U_a$. One of the fundamental identities in Jordan algebras assures that \begin{equation}\label{eq fundamental identity UaUbUa} U_a U_b U_a = U_{U_a(b)}, \hbox{ for all } a,b \hbox{ in a Jordan algebra } M \end{equation} (see \cite[2.4.18]{HOS}). The multiplication operator by an element $a\in M$ will be denoted by $M_a$, that is, $M_a (b)= a \circ b$ ($b\in M$).\smallskip

Henceforth, the powers of an element $a$ in a Jordan algebra $M$ will be denoted as follows:
$$ a^1 =a;\ \  a^{n+1} =a\circ a^n, \quad n\geq 1.$$
If $M$ is unital, we set $a^0=\11b{M}$. An algebra $\mathcal{B}$ is called \emph{power associative} if the subalgebras generated by single elements of $\mathcal{B}$ are associative. In the case of a Jordan algebra $M$ this is equivalent to say that the identity $ a^m\circ a^n=a^{m+n},$ holds for all $a\in M$, $m,n\in \NN.$ It is known that any Jordan algebra is power associative (\cite[Lemma 2.4.5]{HOS}).
\smallskip

By analogy with the associative case, if $M$ is a unital Jordan Banach algebra, the closed subalgebra generated by an element $x\in M$ and the unit is always associative, and hence we can always consider the elements of the form $e^{x}$ in $M$, defined by $\displaystyle e^{x}=\sum_{n=0}^{\infty}\frac{x^n}{n!}$ (cf. \cite[$\S$ 1.1.29]{Cabrera-Rodriguez-vol1}). Let us suppose that $M$ is a unital Jordan Banach subalgebra of an associative unital Banach algebra, and let $x$ be an element in $M$. Take $a$ and $b$ in $M$ such that $a=e^{itx}$ and $b=e^{isx}$, where $t,s\in\RR$. From now on, it will be useful to keep in mind that
\begin{align}\label{eq exp commute}
a\circ b &= e^{itx}\circ e^{isx} \nonumber= \frac12( e^{itx}e^{isx}+ e^{isx}e^{itx})\nonumber \\&= \frac12 (e^{i(t+s)x} + e^{i(t+s)x}) = e^{i(t+s)x} =e^{itx}e^{isx}=ab. \nonumber
\end{align}

An element $a$ in a unital Jordan Banach algebra $M$ is called \emph{invertible} whenever there exists $b\in M$ satisfying $a \circ b = 1$ and $a^2 \circ b = a.$ The element $b$ is unique and it will be denoted by $a^{-1}$ (cf. \cite[3.2.9]{HOS} and \cite[Definition 4.1.2]{Cabrera-Rodriguez-vol1}). We know from \cite[Theorem 4.1.3]{Cabrera-Rodriguez-vol1} that an element $a\in M$ is invertible if and only if $U_a$ is a bijective mapping, and in such a case $U_a^{-1} = U_{a^{-1}}$. \smallskip

A \emph{JB$^*$-algebra} is a complex Jordan Banach algebra $M$ equipped
with an algebra involution $^*$ satisfying  $\|\J a{a}a \|= \|a\|^3$, $a\in
M$ (where $\J a{a}a = U_a (a^*) = 2 (a\circ a^*) \circ a - a^2 \circ a^*$). We know from a result by M.A. Youngson that the involution of every JB$^*$-algebra is an isometry (cf. \cite[Lemma 4]{youngson1978vidav}).\smallskip

A \emph{JB-algebra} is a real Jordan Banach algebra $J$ in which the norm satisfies
the following two axioms for all $a, b\in J$:\begin{enumerate}[$(i)$]\item $\|a^2\| =\|a\|^2$;
\item $\|a^2\|\leq \|a^2 +b^2\|.$
\end{enumerate} The hermitian part, $M_{sa}$, of a JB$^*$-algebra, $M$, is always a JB-algebra. A celebrated theorem due to J.D.M. Wright asserts that, conversely, the complexification of every JB-algebra is a JB$^*$-algebra (see \cite{Wri77}). We refer to the monographs \cite{HOS} and \cite{Cabrera-Rodriguez-vol1} for the basic notions and results in the theory of JB- and JB$^*$-algebras.\smallskip

Every C$^*$-algebra $A$ is a JB$^*$-algebra when equipped with its natural Jordan product $a\circ b =\frac12 (a b +b a)$ and the original norm and involution. Norm-closed Jordan $^*$-subalgebras of C$^*$-algebras are called \emph{JC$^*$-algebras}. JC$^*$-algebras which are also dual Banach spaces are called \emph{JW$^*$-algebras}. Any JW$^*$-algebra is a weak$^*$-closed Jordan $^*$-subalgebra of a von Neumann algebra. \smallskip

We recall that an element $u$ in a unital JB$^*$-algebra $M$ is a \emph{unitary} if it is invertible and its inverse coincides with $u^*$. As in the associative setting, we shall denote by $\mathcal{U}(M)$ the set of all unitary elements in $M$. Let us observe that if a unital C$^*$-algebra is regarded as a JB$^*$-algebra both notions of unitaries coincide. An element $s$ in a unital JB-algebra $J$ is called a \emph{symmetry} if $s^2 =\11b{J}$. If $M$ is a JB$^*$-algebra, the symmetries in $M$ are defined as the symmetries in its self-adjoint part $M_{sa}$.\smallskip

A celebrated result in the theory of JB$^*$-algebras is the so-called Shirshov-Cohn theorem, which affirms that the JB$^*$-subalgebra of a JB$^*$-algebra generated by two self-adjoint elements (and the unit element) is a JC$^*$-algebra, that is, a JB$^*$-subalgebra of some $B(H)$ (cf. \cite[Theorem 7.2.5]{HOS} and \cite[Corollary 2.2]{Wri77}).\smallskip

Two elements $a, b$ in a Jordan algebra $M$ are said to \emph{operator commute} if $$(a\circ c)\circ b= a\circ (c\circ b),$$ for all $c\in M$ (cf. \cite[4.2.4]{HOS}). By the \emph{centre} of $M$ (denoted by $Z(M)$) we mean the set of all elements of $M$ which operator commute with any other element in $M$. Any element in the center is called \emph{central}. \smallskip



%
%
%

A JB$^*$-algebra may admit two different Jordan products compatible with the same norm.  However, when JB$^*$-algebras are regarded as JB$^*$-triples, any surjective linear isometry between them is a triple isomorphism (see \cite[Proposition 5.5]{Ka83}). This fact produces a certain uniqueness of the triple product (see next section for more details). We recall the definition of JB$^*$-triples.\smallskip

A JB$^*$-triple is a complex Banach space $E$ equipped with a continuous triple product $\J ... : E\times E\times E \to E,$ $(a,b,c)\mapsto \{a,b,c\},$ which is bilinear and symmetric in $(a,c)$ and conjugate linear in $b$,
and satisfies the following axioms for all $a,b,x,y\in E$:
\begin{enumerate}[{\rm (a)}] \item $L(a,b) L(x,y) = L(x,y) L(a,b) + L(L(a,b)x,y)
	- L(x,L(b,a)y),$ where $L(a,b):E\to E$ is the operator defined by $L(a,b) x = \J abx;$
\item $L(a,a)$ is a hermitian operator with non-negative spectrum;
\item $\|\{a,a,a\}\| = \|a\|^3$.\end{enumerate}
The definition presented here dates back to 1983 and it was introduced by W. Kaup in \cite{Ka83}.\smallskip

Examples of JB$^*$-triples include all C$^*$-algebras and JB$^*$-algebras with the triple products of the form \begin{equation}\label{eq product operators} \J xyz =\frac12 (x y^* z +z y^*
x),\end{equation}  and \begin{equation}\label{eq product jordan}\J xyz = (x\circ y^*) \circ z + (z\circ y^*)\circ x -
(x\circ z)\circ y^*, \end{equation} respectively. \smallskip

The triple product of every JB$^*$-triple is a non-expansive mapping, that is,
\begin{equation}\label{eq triple product non-expansive} \|\{a,b,c\}\|\leq \|a\| \|b\| \|c\|\ \hbox{ for all } a,b,c \hbox{ (see \cite[Corollary 3]{FriRu86}).}
\end{equation}

Let $E$ be a JB$^*$-triple. Each element $e$ in $E$ satisfying $\J eee=e$ is called a \emph{tripotent}. 
Each tripotent $e\in E$, determines a decomposition of $E,$ known as the \emph{Peirce decomposition}\label{eq Peirce decomposition} associated with $e$, in the form $$E= E_{2} (e) \oplus E_{1} (e) \oplus E_0 (e),$$ where $E_j (e)=\{ x\in E : \J eex = \frac{j}{2}x \}$
for each $j=0,1,2$.\smallskip

Triple products among elements in Peirce subspaces satisfy the following \emph{Peirce arithmetic}: $$\begin{aligned} \J {E_{i}(e)}{E_{j} (e)}{E_{k} (e)} &\subseteq E_{i-j+k} (e)\  \hbox{ if $i-j+k \in \{ 0,1,2\},$ }\\
\J {E_{i}(e)}{E_{j} (e)}{E_{k} (e)} & =\{0\} \hbox{ if $i-j+k \notin \{ 0,1,2\},$ }
\end{aligned}$$ and $\J {E_{2} (e)}{E_{0}(e)}{E} = \J {E_{0} (e)}{E_{2}(e)}{E} =0.$ Consequently, each Peirce subspace $E_j(e)$ is a JB$^*$-subtriple of $E$.\smallskip

The projection $P_{k_{}}(e)$ of $E$ onto $E_{k} (e)$ is called the \emph{Peirce $k$-projection}. It is known that Peirce projections are contractive (cf. \cite[Corollary 1.2]{FriRu85}) and determined by the following identities $P_{2}(e) = Q(e)^2,$ $P_{1}(e) =2(L(e,e)-Q(e)^2),$ and $P_{0}(e) =\hbox{Id}_E - 2 L(e,e) + Q(e)^2,$ where $Q(e):E\to E$ is the conjugate or real linear map defined by $Q(e) (x) =\{e,x,e\}$.\smallskip

It is worth remarking that if $e$ is a tripotent in a JB$^*$-triple $E$, the Peirce 2-subspace $E_2 (e)$ is a unital JB$^*$-algebra with unit $e$,
product $x\circ_e y := \J xey$ and involution $x^{*_e} := \J exe$, respectively (cf. \cite[Theorem 4.1.55]{Cabrera-Rodriguez-vol1}).\label{eq Peirce-2 is a JB-algebra} 
\smallskip

Following standard notation, a tripotent $e$ in in a JB$^*$-triple $E$ is called \emph{unitary} 
if $E_2 (e) = E$. 
\smallskip

\begin{remark}\label{r two different notions of unitaries} The reader should be warned that if a unital JB$^*$-algebra $M$ is regarded as a JB$^*$-triple we have two, a priori, different uses of the word ``unitary''. However, there is no conflict between these two notions because unitary elements in a unital JB$^*$-algebra $M$ are precisely the unitary tripotents in $M$ when the latter is regarded as a JB$^*$-triple (cf. \cite[Proposition 4.3]{BraKaUp78} or \cite[Theorem 4.2.24, Definition 4.2.25 and Fact 4.2.26]{Cabrera-Rodriguez-vol1}). 
\end{remark}


\smallskip

\section{Unitaries in JB$^*$-algebras and inverted Jordan triple products}

Unitary elements in JB$^*$-algebras have been intensively studied for many geometric reasons. As in the setting of C$^*$-algebras, they play a protagonist role in the Wright--Youngson extension of the Russo-Dye theorem for JB$^*$-algebras \cite{WriYou78} (see also \cite[Corollary 3.4.7 and Fact 4.2.39]{Cabrera-Rodriguez-vol1}). Different applications can be found on the study of surjective isometries between JB- and JB$^*$-algebras (see \cite{WriYou78,IsRo95} and \cite[Proposition 4.2.44]{Cabrera-Rodriguez-vol1}). \smallskip

The definition of unitary in a JB$^*$-algebra and its natural connection with the notion of unitary (tripotent) in the setting of JB$^*$-triples has been recalled at the introduction. We shall next revisit some basic properties with the aim of clarifying and make accessible our subsequent arguments.\smallskip

The first result, which has been almost outlined in the introduction, has been borrowed from \cite{Cabrera-Rodriguez-vol1}.

\begin{lemma}\label{l CR Lemma 4.2.41}{\rm\cite[Lemma 4.2.41, Theorem 4.2.28, Corollary 3.4.32]{Cabrera-Rodriguez-vol1}, \cite{WriYou78}, \cite{IsRo95}} Let $M$ be a unital JB$^*$-algebra, and let $u$ be a unitary element in $M$. Then the following statements hold:\begin{enumerate}[$(a)$]\item The Banach space of $M$ becomes a unital JB$^*$-algebra with unit $u$ for the (Jordan) product defined by $x\circ_u y :=U_{x,y}(u^*)=\{x,u,y\}$ and the involution $*_u$ defined by $x^{*_u} :=U_u(x^*)=\{u,x,u\}$. (This JB$^*$-algebra $M(u)= (M,\circ_u,*_u)$ is called the \emph{$u$-isotope} of $M$.)
\item The unitary elements of the JB$^*$-algebras $M$ and $(M,\circ_u,*_u)$ are the same, and they also coincide with the unitary tripotents of $M$ when the latter is regarded as a JB$^*$-triple.
\item The triple product of $M$ satisfies $$\begin{aligned} \J xyz &= (x\circ y^*) \circ z + (z\circ y^*)\circ x -
(x\circ z)\circ y^*\\
&= (x\circ_{u} y^{*_{u}}) \circ_{u} z + (z\circ_{u} y^{*_{u}})\circ_{u} x -
(x\circ_{u} z)\circ_{u} y^{*_{u}},
\end{aligned}$$ for all $x,y,z\in M$. Actually, the previous identities hold when $\circ$ is replaced with any Jordan product on $M$ making the latter a JB$^*$-algebra with the same norm.
\item The mapping $U_u : M\to M$ is a surjective isometry and hence a triple isomorphism. Consequently, $U_u \left( \mathcal{U}(M) \right) = \mathcal{U}(M)$.  Furthermore, the operator $U_u : (M,\circ_{u^*},*_{u^*})\to (M,\circ_u,*_u)$ is a Jordan $^*$-isomorphism.
\end{enumerate}
\end{lemma}

\begin{proof} Statements $(a)$ and $(b)$ can be found in \cite[Lemma 4.2.41]{Cabrera-Rodriguez-vol1} (see also \cite[Proposition 4.3]{BraKaUp78}). Moreover, Theorem 4.2.28~$(vii)$ in \cite{Cabrera-Rodriguez-vol1} assures that the mapping $U_u$ is a surjective linear isometry. The remaining statements are consequences of the fact that a linear bijection between JB$^*$-triples is an isometry if and only if it is a triple isomorphism (cf. \cite[Proposition 5.5]{Ka83} and \cite[Theorem 5.6.57]{Cabrera-Rodriguez-vol2}). Furthermore each unital triple isomorphism between unital JB$^*$-algebras must be a Jordan $^*$-isomorphism.\end{proof}



Let $u$ be a unitary element in a unital C$^*$-algebra $A$. It is known that $\|\textbf{1} -u\|<2$ implies that $u = e^{ih}$ for some $h\in A_{s}$ (see \cite[Exercise 4.6.6]{KR1}). In our next lemma we combine this fact with the Shirshov-Cohn theorem.

\begin{lemma}\label{l untaries at short distance in a unital JB$^*$-algebra} Let $u,v$ be two unitaries in a unital JB$^*$-algebra $M$. Let us suppose that $\|u-v\|= \eta <2.$ Then the following statements hold:\begin{enumerate}[$(a)$] \item There exists a self-adjoint element $h$ in the $u$-isotope JB$^*$-algebra $M(u)= (M,\circ_u,*_u)$ such that $v = e^{ih},$ where the exponential is computed in the JB$^*$-algebra $M(u)$.
\item There exists a unitary $w$ in $M$ satisfying $U_w (u^*) = v$.
\end{enumerate} Moreover, if $\|u-v\|= \eta = \left| 1- e^{it_0} \right| = \sqrt{2}\sqrt{1-\cos(t_0)}$ for some $t_0\in (-\pi,\pi),$ we can further assume that $\|w-u\|,\|w-v\| \leq \sqrt{2}\sqrt{1-\cos(\frac{t_0}{2})}.$
\end{lemma}

\begin{proof} We consider the unital JB$^*$-algebra $M(u)=(M,\circ_u,*_u)$. Let $\mathcal{C}$ denote the JB$^*$-subalgebra of $M(u)$ generated by $v$ and its unit --i.e. $u$--. Let us observe that the product and involution on $\mathcal{C}$ are precisely $\circ_u|_{\mathcal{C}\times \mathcal{C}}$ and $*_u|_{\mathcal{C}}$, respectively. Since $v$ is unitary in $M(u)$ (cf. Lemma \ref{l CR Lemma 4.2.41}$(b)$), the JB$^*$-subalgebra $\mathcal{C}$ must be isometrically Jordan $^*$-isomorphic to a unital commutative C$^*$-algebra, that is, to some $C(\Omega)$ for an appropriate compact Hausdorff space $\Omega,$ and under this identification, $u$ corresponds to the unit  (cf. \cite[3.2.4. The spectral theorem]{HOS} or \cite[Proposition 3.4.2 and Theorem 4.1.3$(v)$]{Cabrera-Rodriguez-vol1}).\smallskip

Since $\|u-v\| <2$, $v$ is a unitary in $\mathcal{C}$ and $u$ is the unit element of the C$^*$-algebra $\mathcal{C} \equiv C(\Omega)$, we can find a self-adjoint element $h\in \mathcal{C}_{sa}$ such that $v = e^{i h}$  (see \cite[Exercise 4.6.6]{KR1}), where 
the exponential is, of course, computed with respect to the structure of $\mathcal{C}$, that is with the product and involution of $M(u)$. This finishes the proof of $(a)$.\smallskip

By setting $w= e^{i \frac{h}{2}}$ we get a unitary element in $\mathcal{C}$ satisfying $w\cdot u^{*_u} \cdot w = w\cdot u \cdot w = v$ (let us observe that the involution of $\mathcal{C}$ is precisely the restriction of $*_{u}$ to $\mathcal{C}$). Let $U_a^{{u}}$ denote the $U$ operator on the unital JB$^*$-algebra $M(u)$ associated with the element $a$. Since $\mathcal{C}$ is a unital JB$^*$-subalgebra of $M(u)$, we deduce that $$v= w\cdot u^{*_u}\cdot w = U_w^{{u}} (u^{*_u}) = \{w,u,w\} = U_{w} (u^*),$$ and clearly $w$ is a unitary in $M$ because it is a unitary in $M(u)$ (cf. Lemma \ref{l CR Lemma 4.2.41}$(b)$ and $(c)$). We have therefore concluded the proof of $(b)$.\smallskip

The final statement is a clear consequence of the identification of $\mathcal{C}$ with $C(\Omega)$ in which $u$ corresponds to the unit and $v$ and $w$ with $e^{i h}$ and $e^{i \frac{h}{2}},$ respectively.
\end{proof}

The next lemma can be deduced by similar arguments to those given in the previous result.

\begin{lemma}\label{l C} Let $u$ and $w$ be unitary elements in a unital JB$^*$-algebra $M$. Suppose $U_{w} (u^*) = u$ and $\|u-w\| < 2$. Then $w = u$.
\end{lemma}

\begin{proof} Let $M(w)=(M,\circ_w,*_w)$ and let $\mathcal{C}$ denote the JB$^*$-subalgebra of $M(w)$ generated by $u$ and its unit. By applying the identification of $\mathcal{C}$ with an appropriate $C(\Omega)$ space as the one given in the proof of the previous lemma, we can identify $w$ with the unit of $C(\Omega)$ and $u$ with a unitary in this commutative unital C$^*$-algebra with $u^{*_{w}} = U_{w} (u^*) = u$ (self-adjoint in $\mathcal{C}$) and $\|u-w\|<2$, which implies that $w=u$.
\end{proof}

We shall gather next some results on isometries between metric groups due to O. Hatori, G. Hirasawa, T. Miura and L Moln{\'a}r \cite{HatHirMiuMol2012}. The conclusions in the just quoted paper provided the tools applied in the study of surjective isometries between the unitary groups of unital C$^*$-algebras in subsequent references \cite{HatMol2012} and \cite{HatMol2014}.\smallskip

Henceforth, let $\mathcal{G}$ be a group and let $(X,d)$ be a non-trivial metric space such that $X$ is a subset of $\mathcal{G}$ and \begin{equation}\label{eq condition p HatoriMiuraMolnar}\hbox{ $y x^{-1} y\in X$ for all $x,y\in X$}
\end{equation} (note that we are not assuming that $X$ is a subgroup of $\mathcal{G}$).

\begin{definition}\label{def condition B} Let us fix $a,b$ in $X$. We shall say that condition $B(a,b)$ holds for $(X,d)$ if the following properties hold:\begin{enumerate}[$(B.1)$]\item For all $x,y\in X$ we have $d(b x^{-1} b, b y^{-1} b) = d(x,y).$
\item There exists a constant $K>1$ satisfying $$d(b x^{-1} b,x)\geq K d(x,b),$$ for all $x\in L_{a,b} =\{x\in X : d(a,x) = d(b a^{-1} b, x) = d(a,b)\}.$
\end{enumerate}
\end{definition}

\begin{definition}\label{def condition C2} Let us fix $a,b\in X$. We shall say that condition $C_1(a,b)$ holds for $(X,d)$ if the following properties hold:\begin{enumerate}[$(C.1)$]\item For every $x\in X$ we have $a x^{-1} b, b x^{-1} a\in X$;
\item $d(a x^{-1} b, a y^{-1} b) = d(x,y)$, for all $x,y\in X$.
\end{enumerate}\smallskip

\noindent We shall say that condition $C_2(a,b)$ holds for $(X,d)$ if there exists $c\in X$ such that $c a^{-1} c =b$ and $d(c x^{-1} c, c y^{-1} c) = d(x,y)$ for all $x,y\in X$.
\end{definition}

An element $x\in X$ is called \emph{2-divisible} if there exists $y\in X$ such that $y^2 =x$. $X$ is called \emph{2-divisible} if every element in $X$ is {2-divisible}. Furthermore, $X$ is a called \emph{2-torsion free} if it contains the unit of $\mathcal{G}$ and the condition $x^2 =1$ with $x\in X$ implies $x =1$.\smallskip

We shall need the following result taken from \cite{HatHirMiuMol2012}.

\begin{theorem}\label{t HatHirMiuMol Thm 24 for metric spaces}\cite[Theorem 2.4]{HatHirMiuMol2012} Let $(X, d_X)$ and $(Y, d_Y )$ be two metric spaces. Pick two points $a,c \in X$. Suppose that $\varphi : X \to X$ is a distance preserving map such that $\varphi(c) = c$ and $\varphi\circ \varphi$ is the identity map on $X$. Let $$L_{a,c} = \{x \in X : d_X (a, x) = d_X(\varphi(a), x) = d_X(a, c)\}.$$
Suppose that that there exists a constant $K > 1$ such that $d_X (\varphi(x), x) \geq K d_X(x, c)$ holds for every $x \in L_{a,c}$. If $\delta$ is a bijective distance preserving map from $X$ onto $Y$, and $\psi$ is a bijective distance preserving map from $Y$ onto itself such that $\psi(\Delta (a)) = \Delta (\varphi(a))$ and $\psi(\Delta (\varphi(a))) = \Delta (a),$ then we have $\psi(\Delta (c)) =\Delta (c).$ $\hfill\Box$
\end{theorem}

Conditions $B(a,b)$, $C_1(a,b)$ and $C_2(a,b)$ are perfectly applied in \cite{HatHirMiuMol2012} (and subsequently in \cite{HatMol2012}) to establish a generalization of the Mazur-Ulam theorem for commutative groups \cite[Corollary 5.1]{HatHirMiuMol2012}, and to present a metric characterization of normed real-linear spaces among commutative metric groups \cite[Corollary 5.4]{HatHirMiuMol2012}. Despite of the tempting title of \cite{HatHirMiuMol2012} for the audience on Jordan structures--i.e. \emph{``Isometries and maps compatible with inverted Jordan triple products on groups''}--, the results in the just quoted reference have not been applied in a proper Jordan setting yet. There are so many handicaps reducing its potential applicability. We are aimed to present a first application in this paper.\smallskip

For the discussion in this paragraph, let $\mathcal{A}$ be a unital JC$^*$-algebra which will be regarded as a JB$^*$-subalgebra of some $B(H)$. Let us observe that the unit of $\mathcal{A}$ must be a projection $\textbf{1}_{\mathcal{A}}$ in $B(H)$, and thus by replacing $H$ with $\textbf{1}_{\mathcal{A}}(H)$, we can always assume that $\mathcal{A}$ and $B(H)$ share the same unit. We shall denote the product of $B(H)$ by mere juxtaposition. The set $\mathcal{U} (\mathcal{A})$ of all unitaries in $\mathcal{A}$ is not in general a subgroup of $\mathcal{U} (B(H))$ --the latter is not even stable under Jordan products--, however $U_u (v) = u v u,$ and $u^*$ lie in $\mathcal{U} (\mathcal{A})$ for all $u,v\in \mathcal{U} (\mathcal{A})$ (cf. Lemma \ref{l CR Lemma 4.2.41}). The set $\mathcal{U} (B(H))$ is a group for its usual product and will be equipped with the distance provided by the operator norm. Conditions of the type $C_1(a,b)$ do not hold for $(\mathcal{U} (\mathcal{A}),\|.\|)$ because products of the form $a x^{-1} b$ do not necessarily lie in $\mathcal{U} (\mathcal{A})$ for all $a,b,x\in \mathcal{U} (\mathcal{A})$. The set $\mathcal{U} (\mathcal{A})$ is not 2-torsion free since $-1\in \mathcal{U} (\mathcal{A})$. Furthermore, the identity $y x^{-1} y = y^2 x^{-1}$ does not necessarily hold for $x,y\in \mathcal{U} (\mathcal{A})$. We have therefore justified that \cite[Corollaries 3.9, 3.10 and 3.11]{HatHirMiuMol2012} cannot be applied in the Jordan setting, even under more favorable hypothesis of working with a JC$^*$-algebra.\smallskip

Let Iso$(Z)$ denote the group of all surjective linear isometries on a Banach space $Z$. Hidden within the proof of \cite[Theorem 6]{HatMol2012}, it is shown that for a complex Banach spaces $Z$, condition $B(a,b)$ is satisfied for elements $a,b$ in Iso$(Z)$ which are at distance strictly smaller than $\frac12$. Let us concretize the exact statement.

\begin{lemma}\label{l HM hiden in Thm 6}\cite[Proof of Theorem 6]{HatMol2012} Let $Z$ be a complex Banach space, let $u,v$ be two elements in Iso$(Z)$ with $\|u-v\|<\frac12$. Then for $K= 2 - 2 \|u-v\|>1,$ the inequality $$\| v w^{-1} v - w \|\geq K \|w-v\|$$ holds for every $w$ in the set $$L_{u,v} = \left\{ w \in \hbox{Iso}(Z) : \|u-w\| = \| v u^{-1} v- w\| = \|u-v\| \right\}.$$
\end{lemma}

The next result is a consequence of the previous lemma in the case in which $M$ is a JC$^*$-algebra by just regarding $M$ as a unital Jordan $^*$-subalgebra of some $B(H)$ with the same unit (we can always see $\mathcal{U}(M)$ inside Iso$(H)$). The existence of exceptional JB$^*$-algebras which cannot be embedded as Jordan $^*$-subalgebras of $B(H)$ (see \cite[Corollary 2.8.5]{HOS}, \cite[Example 3.1.56]{Cabrera-Rodriguez-vol1}), forces us to develop a new argument.

\begin{lemma}\label{l B(a,b) holds for a b close enough} Let $u,v$ be two elements in $\mathcal{U} (M)$, where $M$ is a unital JB$^*$-algebra. Suppose $\|u-v\|<1/2$. Then the Jordan version of condition $B(u,v)$ holds for $\mathcal{U} (M)$, that is, \begin{enumerate}[$(a)$]\item For all $x,y\in \mathcal{U} (M)$ we have $\left\|U_{v}(x^{-1})- U_{v} (y^{-1})\right\|= \|x^*-y^*\| = \|x-y\|.$
\item The constant $K= 2 - 2 \|u-v\|>1$ satisfies that $$ \left\| U_v ( w^*)- w \right\|=\left\| U_v ( w^{-1})- w \right\|\geq K \| w- v\|,$$ for all $w$ in the set $$L_{u,v} =\{w\in \mathcal{U} (M)  : \|u-w\| = \| U_{v} (u^{-1})- w\|= \| U_{v} (u^*)- w\| = \|u-v\|\}.$$
\end{enumerate}
\end{lemma}

\begin{proof} Statement $(a)$ is clear from Lemma \ref{l CR Lemma 4.2.41}$(d)$ and the fact that the involution on $M$ is an isometry.\smallskip

Let us consider the $u$-isotope JB$^*$-algebra $M(u)= (M,\circ_u,*_u)$ of $M$. The $U$-operator on $M(u)$ will be denoted by $U^u$. We fix an element $w\in L_{u,v}$. Since $\|u-w\| = \|u-v\|< \frac12,$ we deduce from Lemma \ref{l untaries at short distance in a unital JB$^*$-algebra} the existence of two self-adjoint elements $h_1,h_2\in M(u)$ such that $v= e^{ih_1}$ and $w= e^{ih_2}$. Let $\mathcal{B}$ denote the JB$^*$-subalgebra of $M(u)$ generated by $u,h_1,h_2$. The Shirshov-Cohn theorem assures the existence of a complex Hilbert space $H$ such that $\mathcal{B}$ is a JB$^*$-subalgebra of $B(H)$ and both share the same unit $u$ (the product of $B(H)$ will be denoted by mere juxtaposition and the involution by $\sharp$). Obviously $u,v,w\in \mathcal{U} (\mathcal{B})\subseteq \mathcal{U} (B(H))\subseteq \hbox{Iso}(H)$ with $\|u-w\|_{\mathcal{B}} = \|u-w\|_{M}= \|u-v\|_{M}= \|u-v\|_{\mathcal{B}}$. Let us compute $$\begin{aligned} U_{v}^u (u) &= 2 (v\circ_u u)\circ_u v - (v\circ_u v)\circ_u u 
\\
&= 2 v\circ_u v - v\circ_u v = v\circ_u v = \{v,u,v\} = U_v(u^*),
\end{aligned}$$
 and
$$\begin{aligned} \| v u^{-1} v- w\|_{_{B(H)}}&= \| v u^{-1} v- w\|_{\mathcal{B}} = \| U_{v}^u (u^{\sharp}) - w\|_{\mathcal{B}} = \| U_{v}^u (u^{*_u}) - w\|_{\mathcal{B}} \\
&= \| U_{v}^u (u) - w\|_{\mathcal{B}} = \| U_{v} (u^*)- w\|_{\mathcal{B}}= \| U_{v} (u^*)- w\|_{M}.
\end{aligned} $$

Lemma \ref{l HM hiden in Thm 6} proves that for $K= 2 - 2 \|u-v\|>1$ we have $$\left\| U^u_v ( w^{*_u})- w \right\|_{\mathcal{B}} = \left\| U^u_v ( w^{-1})- w \right\|_{\mathcal{B}} = \| v w^{-1} v - w \|_{_{B(H)}} \geq K \|w-v\|_{\mathcal{B}} = K \|w-v\|.$$
On the other hand, by the uniqueness of the triple product (see \cite[Proposition 5.5]{Ka83} or Lemma \ref{l CR Lemma 4.2.41}$(c)$) we have
$$ U^u_v ( w^{*_u}) = \{v,w,v\}_{\mathcal{B}} =  \{v,w,v\}_{M(u)} =  \{v,w,v\}_{M} = U_v(w^*).$$ All together gives $$\begin{aligned}\left\| U_v ( w^{-1})- w  \right\|&=\left\|U_v ( w^*)- w \right\|_{M} =  \left\| U_v ( w^*)- w \right\|_{\mathcal{B}} \\
&= \left\| U^u_v ( w^{*_u})- w \right\|_{\mathcal{B}}\geq K \| w- v\|,
\end{aligned}$$ which completes the proof.
\end{proof}

We can now establish a key result for our goals.

\begin{theorem}\label{t Delta preserves inverted triple products in UM} Let $\Delta: \mathcal{U} (M)\to \mathcal{U} (N)$ be a surjective isometry, where $M$ and $N$ are unital JB$^*$-algebras. Suppose $u,v\in \mathcal{U} (M)$ with $\|u-v\|<\frac12$. Then the following statements are true:
\begin{enumerate}[$(1)$]\item The Jordan version of condition $B(u,v)$ holds for $\mathcal{U} (M)$;
\item The Jordan version of condition $C_2 (\Delta(u),\Delta(U_v (u^{*})))$ holds for $\mathcal{U} (N)$;
\item The identity $\Delta(U_{v} (u^{*})) =\Delta(U_{v} (u^{-1})) = U_{\Delta(v)} (\Delta(u)^*)= U_{\Delta(v)} (\Delta(u)^{-1})$ holds.
\end{enumerate}
\end{theorem}

\begin{proof} Statement $(1)$ follows from Lemma \ref{l B(a,b) holds for a b close enough}.\smallskip

$(2)$ By hypotheses $\| \Delta(u) - \Delta(v)\| =\|u-v\|<\frac12$. Let $\mathcal{B}$ denote the JB$^*$-subalgebra of the $u$-isotope $M(u)$ generated by $u$ and $v$. We shall denote by $U^u$ the $U$-operator in $M(u)$. Again, by the Shirshov-Cohn theorem we can find a complex Hilbert space $H$ such that $\mathcal{B}$ is a JB$^*$-subalgebra of $B(H)$ and $u$ is the unit of $B(H)$. The product of $B(H)$ will be denoted by mere juxtaposition and the involution by $\sharp$. In this case we have
$$\begin{aligned} \| \Delta(u) - \Delta(U_v (u^*))\| &=  \| u - U_v (u^*)\| =   \| u - \{ v, u,v\}\| = \| u - \{ v, u,v\}_{\mathcal{B}}\|_{\mathcal{B}}  \\
&=\| u - U^u_v (u^{*_u})\|_{\mathcal{B}} =  \| u - U_v^u (u^{*_{u}})\|_{B(H)} =  \| u - v u^{\sharp} v\|_{B(H)} \\
&=  \| u - v v\|_{_{B(H)}} =  \| v^{\sharp} - v\|_{_{B(H)}} \leq   \| v^{\sharp} - u \|_{_{B(H)}} + \| u- v\|_{_{B(H)}} \\
&= 2 \| u - v\|_{B(H)}= 2 \| u - v\|_{\mathcal{B}} =  2 \| u - v\| < 1
\end{aligned}$$ (cf. Lemma \ref{l CR Lemma 4.2.41}$(c)$).\smallskip

Lemma \ref{l untaries at short distance in a unital JB$^*$-algebra}$(b)$ assures the existence of $w\in \mathcal{U} (N)$ satisfying \begin{equation}\label{eq birth of w in the proof of the theorem} U_w (\Delta(u)^*) = \Delta\left( U_v (u^*) \right), \hbox{ and } \|w- \Delta(v) \| < 1 \hbox{ (or smaller)}.
 \end{equation} This shows that the Jordan version of $C_2 (\Delta(u),\Delta(U_u (v^{*})))$ holds for $\mathcal{U} (N)$ because the remaining requirement, i.e., $$\| U_w (x^{-1}) -U_w (y^{-1})\| = \|x^{-1}- y^{-1}\| = \|x^* - y^* \| = \|x-y\|$$ holds for all $x,y\in \mathcal{U} (N)$ (cf. Lemma \ref{l CR Lemma 4.2.41}$(d)$).\smallskip

$(3)$ Let $w\in \mathcal{U} (N)$ be the element found in the proof of $(2)$. We define a couple of mappings $\varphi: \mathcal{U} (M)\to \mathcal{U} (M)$ and $\psi: \mathcal{U} (N)\to \mathcal{U} (N)$ given by $\varphi (x) := U_{v} (x^{-1}) = U_{v} (x^{*})$ and $\psi (y) := U_{w} (y^{-1}) = U_{w} (y^{*})$, respectively. Clearly, $\varphi$ and $\psi$ are distance preserving bijections (cf. Lemma \ref{l CR Lemma 4.2.41}$(d)$).\smallskip

It is clear that $\varphi (v) = v$ and $\varphi\circ \varphi$ is the identity mapping on $\mathcal{U} (M)$. Furthermore, by \eqref{eq birth of w in the proof of the theorem} we have $$\begin{aligned}
\psi (\Delta(u)) &= U_w (\Delta(u)^*) = \Delta\left( U_v (u^*) \right) =\Delta\left( \varphi (u) \right) \\
\psi \left(\Delta\left( \varphi (u) \right)\right) &= U_{w} \left(\Delta \left( U_{v} (u^*)^*\right) \right)= U_{w^*} \left(\Delta \left( U_{v} (u^*)\right)\right)^* = \Delta(u)^{**} = \Delta (u).
\end{aligned}$$

Since by $(1)$ the Jordan version of condition $B(u,v)$ holds for $\mathcal{U} (M)$, and by $(2)$ the Jordan version of $C_2 (\Delta(u),\Delta(U_v (u^{*})))$ holds for $\mathcal{U} (N)$, we can see that all the hypotheses of Theorem \ref{t HatHirMiuMol Thm 24 for metric spaces} \cite[Theorem 2.4]{HatHirMiuMol2012} are satisfied. We deduce from the just quoted theorem that $U_{w} (\Delta(v)^*) = \Delta(v),$ and since $\|w -\Delta(v)\|<2$ (cf. \eqref{eq birth of w in the proof of the theorem}), Lemma \ref{l C} guarantees that $w=\Delta(v)$ and hence
$\Delta(U_{v} (u^{*}))  = U_{\Delta(v)} (\Delta(u)^*)$ as desired (see \eqref{eq birth of w in the proof of the theorem}).
\end{proof}

\section{Surjective isometries between sets of unitaries}

In this section we shall try to find a precise description of the surjective isometries between the sets of unitaries in two unital JB$^*$-algebras. Our first goal is to find conditions under which any such surjective isometry can be extended to a surjective real linear isometry between these JB$^*$-algebras.\smallskip

We recall that a \emph{one-parameter group} of bounded linear operators on a Banach space $Z$ is a mapping $\mathbb{R} \to B(Z),$ $t\mapsto E(t)$ satisfying $E(0) =I$ and $E({t+s}) = E({s}) E({t}),$ for all $s,t\in \mathbb{R}$. A one-parameter group $\{ E(t) : t\in \mathbb{R}\}$ is uniformly continuous at the origin if $\displaystyle \lim_{t\to 0} \| E(t) -I\| =0$. It is known that being uniformly continuous at zero is equivalent to the existence of a bounded linear operator $R\in B(Z)$ such that $E(t) =e^{t R}$ for all $t\in \mathbb{R}$, where the exponential is computed in the Banach algebra $B(Z)$ (see, for example, \cite[Proposition 3.1.1]{BratRob1987}). A one-parameter group on $\{E(t) : t\in \mathbb{R}\}$ on a complex Hilbert space $H$ is called \emph{strongly continuous} if for each $\xi$ in $H$ the mapping $t\mapsto  E({t}) (\xi)$ is continuous (\cite[Definition 5.3, Chapter X]{Conway}). A \emph{one-parameter unitary group} on $H$ is a one-parameter group on $H$ such that $E(t)$ is a unitary element for each $t\in \mathbb{R}$. \smallskip

The celebrated Stone's one-parameter theorem affirms that for each strongly continuous one-parameter unitary group $\{E(t) : t\in \RR \}$ on a complex Hilbert space $H$ there exists a self-adjoint operator $h\in B(H)$ such that $E(t)=e^{i t h}$, for every $t\in \RR$ (\cite[5.6, Chapter X]{Conway}).\smallskip

We recall that a \emph{triple derivation} on a JB$^*$-triple $E$ is a linear mapping $\delta: E\to E$ satisfying a ternary version of Leibniz' rule $$\delta \{a,b,c\} = \{\delta(a),b,c\} + \{a, \delta(b), c\} + \{a,b,\delta(c)\}, \ \ (a,b,c\in E).$$ We shall apply that every triple derivation is automatically continuous (see \cite[Corollary 2.2]{BarFri90}). If $\delta: M\to M$ is a triple derivation on a unital JB$^*$-algebra, it is known that $\delta(\11bM)^* = - \delta(\11bM)$, that is, $i \delta(\11bM)\in M_{sa}$ (cf. \cite[Proof of Lemma 1]{HoMarPeRu}).\smallskip

The existence of exceptional JB$^*$-algebra which cannot be represented inside a $B(H)$ space, adds extra difficulties to apply Stone's one-parameter theorem. The study of uniformly continuous one-parameter groups of surjective isometries (i.e. triple isomorphisms), Jordan $^*$-isomorphisms and orthogonality preserving operators on JB$^*$-algebras has been recently initiated in \cite{GarPe20}. We complement the results in the just quoted reference with the next result, which is a Jordan version of Stone's theorem for uniformly continuous unitary one-parameter groups.

\begin{theorem}\label{t Jordan unitary groups version of Stone's theorem} Let $M$ be a unital JB$^*$-algebra. Suppose $\{u(t):t\in \mathbb{R}\}$ is a family in $\mathcal{U} (M)$ satisfying $u(0) =\textbf{1},$ and $U_{u(t)} (u(s)) = u(2 t +s),$ for all $t,s\in \mathbb{R}$. We also assume that the mapping $t\mapsto u(t)$ is continuous. Then there exists $h\in M_{sa}$ such that $u(t) = e^{i t h}$ for all $t\in \mathbb{R}$.
\end{theorem}

\begin{proof} We shall first prove that \begin{equation}\label{eq first u is a one-parameter group}\hbox{$u(s+t) = u(t) \circ u(s)$ for all $t,s\in \mathbb{R}$.}
 \end{equation} Fix a real $t$, it follows from the hypothesis that   $$u(t)^2=U_{u(t)}(u(0))=u(2 t), \hbox{ and } u(t)^3=U_{u(t)} (u(t))= u(3 t).$$
Arguing by induction on $n$, it can be established that $u(t)^n = u(n t)$ for all $n\in \mathbb{N}$, $t\in \mathbb{R}$. Indeed, for any integer $n\geq 4$, by the induction hypothesis
$$u(t)^n= U_{u(t)} (u(t)^{n-2})= U_{u(t)}(u((n-2) t)) = u(n t).$$

The identity $U_{u(t)} U_{u(-t)} U_{u(t)} = U_{u(t)}$ (see \eqref{eq fundamental identity UaUbUa}) together with the fact that $U_{u(t)}$ is invertible in $B(M)$ proves that $U_{u(t)^*} =U_{u(t)}^{-1} = U_{u(-t)}$ and thus $u(t)^* = u(t)^{-1}  = u(-t)$ for all $t\in \mathbb{R}$ (cf. \cite[Theorem 4.1.3]{Cabrera-Rodriguez-vol1}).\smallskip

Given an integer $n\leq 0$ we observe that $u(t)^n={(u(t)^{-n})^{-1}}={(u(t)^{-n})}^*$ for every ${t\in \RR}$, and thus $$u(t)^n = (u(t)^{-n})^* = u((-n) t)^* = u(n t).$$ We have therefore shown that
\begin{equation}\label{eq u(t) preserves powers} u(t)^n = u(n t), \hbox{ for all } {t\in \RR} \hbox{ and } n\in \mathbb{Z}.
\end{equation}

By the continuity of the mapping $t\mapsto u(t)$, in order to prove \eqref{eq first u is a one-parameter group} it suffices to show that the identity $u(r+r')= u(r)\circ u(r')$ holds for any rational numbers $r$ and $r'$. Therefore, let us take $r=n/m$ and $r'=n'/m'$, with $n,n'\in \mathbb{Z}$ and $m,m'\in \mathbb{N}$. By a couple of applications of \eqref{eq u(t) preserves powers} and the power associativity of $M$ (see \cite[Lemma 2.4.5]{HOS}) we have
$$\begin{aligned} u(r+r') &= u\left(\frac{nm'+mn'}{mm'}\right) = u\left(\frac{1}{mm'}\right)^{nm'+mn'}= u\left(\frac{1}{mm'}\right)^{n m'} \circ u\left(\frac{1}{mm'}\right)^{m n'}\\
&= u\left(\frac{n m'}{mm'}\right) \circ u\left(\frac{m n'}{mm'}\right) =u(r)\circ u(r'),
\end{aligned}$$ as desired.\smallskip

Let us define a mapping $\Phi: \mathbb{R}\to \hbox{Iso} (M)$, $t\mapsto \Phi(t) = U_{u(t)}$. Clearly, $\Phi$ is continuous with $\Phi(0) = Id_{M}$. By applying the fundamental identity \eqref{eq fundamental identity UaUbUa} one sees that $$\Phi(t) \Phi(s) \Phi(t) =  U_{u(t)} U_{u(s)} U_{u(t)} = U_{U_{u(t)}(u(s))} = U_{u(2t+s)} = \Phi(2 t+s),$$ for all $s,t\in \mathbb{R}$. It then follows that $\Phi (t)^2 = \Phi(t) \Phi(0) \Phi(t) = \Phi(2 t)$, $\Phi (3 t) = \Phi(t)^3$, and by induction on $n(\geq 3)$, $\Phi( n t) =\Phi(2t + (n-2) t) =\Phi(t) \Phi((n-2) t) \Phi(t)= \Phi(t) \Phi(t)^{n-2} \Phi(t) = \Phi (t)^n$ for all $t\in \mathbb{R}$, $n\in \mathbb{N}$. Therefore $\Phi( n t) = \Phi(t)^{n}$ for all $t\in \mathbb{R}$, $n\in \mathbb{Z}$.  \smallskip

Since, for each real $t$, $\Phi (t) \Phi (-t) \Phi (t) = \Phi (t)$ and $\Phi(t)\in \hbox{Iso} (M)$ we can deduce that $\Phi(-t) = \Phi(t)^{-1}$ for all $t\in \mathbb{R}$. It follows that, for a negative integer $n$ and each real $t$, we have $$\Phi(n t) = \Phi ((-n) (-t))= \Phi(-t)^{-n} = \Phi (t)^n,$$ an identity which then holds for all $n\in \mathbb{Z}$.\smallskip

We claim that $t\mapsto \Phi (t)$ is a one-parameter group of surjective isometries on $M$. Let us fix two rational numbers $\frac{n}{m}, \frac{n'}{m'}$ with $m,m'\in \mathbb{N}$, $n,n'\in \mathbb{Z}$. It follows from the above properties that $$\Phi \left(\frac{n}{m} +\frac{n'}{m'}\right) = \Phi \left(\frac{n m' + n' m}{m m'}\right) = \Phi \left(\frac{1}{m m'}\right)^{n m' + n' m}  $$ $$= \Phi \left(\frac{1}{m m'}\right)^{n m' }  \Phi \left(\frac{1}{m m'}\right)^{n' m} = \Phi \left(\frac{n}{m}\right)  \Phi \left(\frac{n'}{m'}\right).$$ It follows from the continuity of $\Phi$ that $$\Phi (t+s) = \Phi(t) \Phi(s), \hbox{ for all } t,s\in \mathbb{R},$$ that is, $\Phi (t)$ is a uniformly continuous one-parameter group of surjective isometries on $M$. By \cite[Lemma 3.1]{GarPe20} there exists a triple derivation $\delta: M\to M$ satisfying $\displaystyle \Phi (t) = e^{t \delta}= \sum_{n=0}^{\infty} \frac{t^n}{n!} \delta^n$ for all $t\in \mathbb{R}$, where the exponential is computed in $B(M)$.\smallskip

Let us observe that $w(t):=\Phi(t) (\textbf{1}) = U_{u(t)} (\textbf{1}) = u(t)^2 = U_{u(t)} (u(0))= u(2 t)$ for all $t\in \mathbb{R}$, and hence by \eqref{eq first u is a one-parameter group} $$w(t+s)=\Phi(t+s) (\textbf{1}) = u(2(t+s)) = u(2 t) \circ u(2 s)$$ $$ = \Phi(t) (\textbf{1}) \circ \Phi(s) (\textbf{1})= w(t) \circ w(s),$$ for all $s,t\in \mathbb{R}$. We shall next show that $$\delta(\textbf{1})^n = \delta^n(\textbf{1}), \hbox{ for all natural } n.$$ Since $w(t+s) = w(t) \circ w(s)$, by taking derivatives in $t$ at $t=0$ we get $$\sum_{n=1}^{\infty} \frac{s^{n-1}}{(n-1)!} \delta^{n} (\textbf{1})= \frac{\partial}{\partial t}_{|_{{t=0}}} w(t+s) = \frac{\partial}{\partial t}_{|_{t=0}} w(t)\circ w(s) = \delta(\textbf{1}) \circ w(s),$$ for all $s\in \mathbb{R}$.  Taking a new derivative in $s$ at $s=0$ we have $$\delta^2(\textbf{1}) = \frac{\partial}{\partial s}_{|_{s=0}} \frac{\partial}{\partial t}_{|_{t=0}} w(t+s) 
= \delta(\textbf{1}) \circ \frac{\partial}{\partial s}_{|_{s=0}} w(s)$$
$$=  \delta(\textbf{1}) \circ (\delta(\textbf{1}) \circ w(0))= M_{\delta(\textbf{1})}^{2} (w(0)) = \delta(\textbf{1})^2.$$ Similarly,  $$\delta^n (\textbf{1}) = \frac{\partial^{n-1}}{\partial s^{n-1}}_{|_{s=0}} \frac{\partial}{\partial t}_{|_{t=0}} w(t+s)= M_{\delta(\textbf{1})}^{n} (w(0)) = \delta(\textbf{1})^n,$$ which gives the desired statement.\smallskip

It follows from the above identities that $$u(t) =\Phi(\frac{t}{2}) = e^{\frac{t}{2} \delta} (\textbf{1}) = \sum_{n=0}^{\infty} \frac{t^n}{2^n n!} \delta^n (\textbf{1}) = \sum_{n=0}^{\infty} \frac{t^n}{ 2^n n!} \delta(\textbf{1})^n = e^{t \frac{\delta(\textbf{1})}{2} }, \hbox{ for all } t\in \mathbb{R},$$ where, as we commented before this proposition, $\delta(\textbf{1})^* = -\delta(\textbf{1})$ in $M$  (cf. \cite[Proof of Lemma 1]{HoMarPeRu}).
\end{proof}

We continue by enunciating a variant of an argument which has been applied in several cases before.

\begin{remark}\label{r Jordan starisomorph from a unitary in JBWstar alg} Suppose $u$ is a unitary element in a unital JB$^*$-algebra $M$ such that $\|\textbf{1}-u\|<2$. By Lemma \ref{l untaries at short distance in a unital JB$^*$-algebra}$(a)$ we can find a self-adjoint element $h\in M_{sa}$ satisfying $u = e^{ih}$. Let us consider the unitary $\omega = e^{- i \frac{h}{2}}\in \mathcal{U} (M)$ and the mapping $U_{\omega} : M\to M$. Let us observe that $U_\omega(u) = \textbf{1}$ (just apply that $u$ and $\omega$ operator commute by definition). Since $\omega$ is unitary in $M$, and hence $U_{\omega}^{-1} = U_{\omega^*}$ (cf. \cite[Theorem 4.1.3]{Cabrera-Rodriguez-vol1}) we can conclude from \eqref{eq triple product non-expansive} that $U_\omega : M(u) =(M,\circ_u, *_{u})\to M$ is a unital surjective isometry (see also Lemma \ref{l CR Lemma 4.2.41}$(d)$ or \cite[Theorem 4.2.28]{Cabrera-Rodriguez-vol1} for a direct argument). We can therefore conclude from Theorem 6 in \cite{WriYou78} that $U_\omega : M(u)=(M,\circ_u, *_{u})\to M$ is a Jordan $^*$-isomorphism.\smallskip

As in the case of unital C$^*$-algebras (cf. the discussion preceding Proposition 4.4.10 in \cite{KR1}), not each unitary element of a unital JB$^*$-algebra $M$ is of the form $e^{ih}$ for some $h\in M_{sa}$. However, if we assume that $M$ is a JBW$^*$-algebra the conclusion is different. Let $u$ be a unitary element in a JBW$^*$-algebra $M$. Let $\mathcal{W}$ denote the JBW$^*$-subalgebra of $M$ generated by $u,$ $u^*$ and the unit of $M$. Clearly $\mathcal{W}$ is an associative JBW$^*$-algebra (cf. \cite[Theorem 3.2.2, Remark 3.2.3 and Theorem 4.4.16]{HOS}), and we can therefore assume that $\mathcal{W}$ is a commutative von Neumann algebra. Theorem 5.2.5 in \cite{KR1} implies the existence of an element $h\in \mathcal{W}_{sa}\subseteq M_{sa}$ such that $e^{i h} = u$ (in $\mathcal{W}$ and also in $M$). We therefore have \begin{equation}\label{eq unitaries are exponentials in JBW*-algebras} \mathcal{U} (M) = \{e^{i h} : h\in M_{sa}\}
 \end{equation} for all JBW$^*$-algebra $M$.
\end{remark}

The next lemma, which is a Jordan version of \cite[Lemma 7]{HatMol2012}, will be required later.

\begin{lemma}\label{l HM Lemma 7 Jordan} Let $M$ and $N$ be two unital JB$^*$-algebras. Let $\{u_{k}: 0\leq k\leq 2^n\}$ be a subset of $\mathcal{U}(M)$ and let $\Phi: \mathcal{U}(M)\to \mathcal{U}(N)$ be a mapping such that $U_{u_{k+1}} (u_k^*) = u_{k+2}$ and $$\Phi (U_{u_{k+1}} (u_k^*)) = U_{\Phi(u_{k+1})} (\Phi (u_k)^*),$$ for all $0\leq k\leq 2^n -2$. Then $U_{u_{2^{n-1}}} (u_0^*) = u_{2^n}$ and $$\Phi\left(U_{u_{2^{n-1}}} (u_0^*) \right) = U_{\Phi(u_{2^{n-1}})} \Phi(u_0)^*. $$
\end{lemma}

\begin{proof} We shall argue by induction on $n$. The statement is clear for $n=1$. Suppose that our statement is true for every family with $2^n+1$ elements satisfying the conditions above. Let $\{w_{k}: 0\leq k\leq 2^{n+1}\}$ be a subset of $\mathcal{U}(M)$ such that $U_{w_{k+1}} (w_k^*) = w_{k+2}$ and $$\Phi (U_{w_{k+1}} (w_k^*)) = U_{\Phi(w_{k+1})} (\Phi (w_k)^*),$$ for all $0\leq k\leq 2^{n+1} -2$. Set $u_k = w_{2k}$. We shall next show that we can apply the induction hypothesis to the family $\{u_{k}: 0\leq k\leq 2^{n}\}\subseteq \mathcal{U} (M)$. Fix $0\leq k\leq 2^n -2$,
$$\begin{aligned} u_{k+2}  &= w_{2 k +4} = U_{w_{2k +3}} (w_{2k +2}^*) = U_{U_{w_{2k +2}} (w_{2k +1}^*) } (w_{2k +2}^*) \\
&= U_{w_{2k +2}} U_{w_{2k +1}^*} U_{w_{2k +2}} (w_{2k +2}^*) = U_{w_{2k +2}} U_{w_{2k +1}^*} (w_{2k +2})\\
&= U_{w_{2k +2}} U_{w_{2k +1}^*} U_{w_{2k +1}} (w_{2k}^*) = U_{w_{2k +2}}  (w_{2k}^*) = U_{u_{k +1}}  (u_{k}^*),
\end{aligned} $$ where in the fourth equality we applied the identity \eqref{eq fundamental identity UaUbUa}.\smallskip

On the other hand, the previous identities and the induction hypothesis also give $$\begin{aligned}\Phi (U_{u_{k+1}} (u_k^*))& = \Phi (U_{w_{2k+2}} (w_{2k}^*)) = \Phi ( U_{w_{2k +3}} (w_{2k +2}^*) ) = U_{\Phi(w_{2 k+3})} (\Phi (w_{2k+2})^*) \\
&= U_{\Phi( U_{w_{2 k+2}} (w_{2k+1}^*))} (\Phi (w_{2k+2})^*) = U_{ U_{\Phi(w_{2 k+2})} (\Phi(w_{2k+1})^*)} (\Phi (w_{2k+2})^*)\\
& = U_{\Phi(w_{2 k+2})} U_{\Phi(w_{2k+1})^*} U_{\Phi(w_{2 k+2})} (\Phi (w_{2k+2})^*) \\
&= U_{\Phi(w_{2 k+2})} U_{\Phi(w_{2k+1})^*} (\Phi(w_{2 k+2})) \\
&= U_{\Phi(w_{2 k+2})} U_{\Phi(w_{2k+1})^*} (\Phi(U_{w_{2 k+1}} (w_{2k}^*))) \\
&= U_{\Phi(w_{2 k+2})} U_{\Phi(w_{2k+1})^*} U_{\Phi(w_{2 k+1})} (\Phi(w_{2k})^*)\\
&= U_{\Phi(w_{2 k+2})} (\Phi(w_{2k})^*) = U_{\Phi(u_{k+1})} (\Phi (u_k)^*),
 \end{aligned}$$ where in the sixth equality we applied the identity \eqref{eq fundamental identity UaUbUa}.\smallskip

It follows from the induction hypothesis that $$U_{w_{2^{n}}} (w_0^*) = U_{u_{2^{n-1}}} (u_0^*) = u_{2^n} = w_{2^{n+1}}$$ and $$\Phi\left(U_{w_{2^{n}}} (w_0^*) \right) =\Phi\left(U_{u_{2^{n-1}}} (u_0^*) \right) = U_{\Phi(u_{2^{n-1}})} \Phi(u_0)^*= U_{\Phi(w_{2^{n}})} \Phi(w_0)^*,$$ which finishes the induction argument.
\end{proof}

We are now ready to establish our first main result, which asserts that, under some mild conditions, for each surjective isometry $\Delta$ between the unitary sets of two unital JB$^*$-algebras $M$ and $N$ we can find a surjective real linear isometry $\Psi:M\to N$ which coincides with $\Delta$ on the subset $e^{i M_{sa}}$.

\begin{theorem}\label{t HM for unitary sets of unitary JB*-algebras}
Let $\Delta: \mathcal{U} (M)\to \mathcal{U} (N)$ be a surjective isometry, where $M$ and $N$ are two unital JB$^*$-algebras. Suppose that one of the following holds:
\begin{enumerate}[$(1)$]\item $\|\11b{N}-\Delta(\11b{M})\|<2$;
\item There exists a unitary $\omega_0$ in $N$ such that $U_{\omega_0} (\Delta(\11bM)) = \11bN.$
\end{enumerate} Then there exists a unitary $\omega$ in $N$ satisfying $$\Delta(e^{i M_{sa}})=U_{\omega^*}(e^{i N_{sa}}).$$ Furthermore, there exists a central projection $p\in N$ and a Jordan $^*$-isomorphism $\Phi:M\to N$ such that $$\begin{aligned}\Delta(e^{ih}) &=  U_{\omega^*}\left( p \circ \Phi(e^{ih}) \right) + U_{\omega^*}\left( (\11b{N}-p) \circ\Phi( e^{ih})^*\right)\\
&=  P_2(U_{\omega^*}( p)) U_{\omega^*}(\Phi(e^{ih}))  + P_2(U_{\omega^*}(\11b{N} -p)) U_{\omega^*}(\Phi( (e^{ih})^*)),
\end{aligned}$$ for all $h\in M_{sa}$. Consequently, the restriction $\Delta|_{e^{i M_{sa}}}$ admits a (unique) extension to a surjective real linear isometry from $M$ onto $N$.
\end{theorem}

\begin{proof} If $\Delta$ satisfies $(1)$, by Remark \ref{r Jordan starisomorph from a unitary in JBWstar alg} (see also Lemma \ref{l untaries at short distance in a unital JB$^*$-algebra}$(a)$) there exists a unitary $\omega$ in $N$ such that $U_{\omega} (\Delta(\11bM)) = \11bN$, and $U_\omega$ is an isometric Jordan $^*$-isomorphism from the $\Delta(\textbf{1})$-isotope $N(\Delta(\textbf{1}))$ onto $N$. Since $\mathcal{U} \left( N,\circ_{\Delta(\11b{M})}, *_{\Delta(\11b{M})}\right) $  $= $ $\mathcal{U}(N)$ (see Lemma \ref{l CR Lemma 4.2.41}) we have $U_\omega(\mathcal{U}(N))=\mathcal{U}(N)$. In case that $(2)$ holds we take $\omega=\omega_0$. \smallskip

The surjective isometry $\Delta_0: \mathcal{U} (M)\to \mathcal{U} (N),$ $\Delta_0(u)=U_\omega(\Delta(u)),$ satisfies $\Delta_0(\11b{M})=\11b{N}$, that is, $\Delta_0$ is a unital surjective isometry between the unitary sets of $M$ and $N$.\smallskip

Fix $h\in M_{sa}$. We consider the continuous mapping $\mathbb{R}\to \mathcal{U} (M)$, $E_h(t)=e^{ith}$. Let $u$ and $v$ be two arbitrary unitary elements in $E_h(\RR)$ (that is, $u=e^{ith}$ and $v=e^{ish}$, for some $t,s\in \RR$). Choose now a positive integer $m$ such that
$\displaystyle e^{\frac{\| i(s-t)h \|_{_{M}}}{2^m}}-1 < \frac{1}{2}.$
It follows from this assumption that
\begin{equation}\label{eq m}
\| e^{\frac{i(s-t)h}{2^m}}-\11b{M} \|_{_{M}}\leq e^{\frac{\| i(s-t)h \|_{_{M}}}{2^m}}-1 < \frac{1}{2}.
\end{equation}
Let us define the family $\{u_k: 0\leq k \leq 2^{m+1}\}$, where $$u_k=u\circ e^{\frac{ik(s-t)h}{2^m}}, \quad 0\leq k \leq 2^{m+1}.$$
Since $u$ and $e^{\frac{ik(s-t)h}{2^m}}$ operator commute in $M,$ the element $u_k$ is a unitary in $M$ for every $k\in \{0,\dots,2^{m+1}\}$. Clearly, $u_0=u,$ $u_{2^m}=v$ and $u_{2^{m+1}}=U_{v} (u^*)$.\smallskip

Any two elements in the family $\{u_k\}_{k=0}^{2^{m+1}}$ operator commute, thus it is not hard to see that for any $0\leq k \leq 2^{m+1}-2$, $$U_{u_{k+1}}(u_k^{-1})=U_{u_{k+1}}(u_k^*)= u_{k+2}.$$
On the other hand, by our election of $m$ in (\ref{eq m}), we have
$$
\begin{aligned}
\|u_{k+1}-u_k\|_{_{M}}&= \| u\circ e^{\frac{i(k+1)(s-t)h}{2^m}}- u\circ e^{\frac{ik(s-t)h}{2^m}} \|_{_{M}} \\ 
&\leq  \|u\|_{_{M}}\| e^{\frac{ik(s-t)h}{2^m}}\circ e^{\frac{i(s-t)h}{2^m}} - e^{\frac{ik(s-t)h}{2^m}} \|_{_{M}}\\ &= \| e^{\frac{ik(s-t)h}{2^m}}\circ \left( e^{\frac{i(s-t)h}{2^m}}-\11b{M}\right)  \|_{_{M}}\\ & \leq \|e^{\frac{ik(s-t)h}{2^m}}\|_{_{M}}\| e^{\frac{i(s-t)h}{2^m}}-\11b{M} \|_{_{M}} 
\leq e^{\frac{\| i(s-t)h \|_{_{M}}}{2^m}}-1 < \frac{1}{2},
\end{aligned}
$$ for any $0\leq k \leq 2^{m+1}-1$.
Theorem \ref{t Delta preserves inverted triple products in UM}(3) affirms that the identity $$\Delta(U_{u_{k+1}} (u_k^{-1})) = \Delta(U_{u_{k+1}} (u_k^{*}))= U_{\Delta(u_{k+1})} (\Delta(u_k)^*)= U_{\Delta(u_{k+1})} (\Delta(u_k)^{-1}),$$ holds for every $0\leq k \leq 2^{m+1}-1$.\smallskip

Lemma \ref{l HM Lemma 7 Jordan} applied to $\Delta$ and the family $\{u_{k} : 0\leq k\leq 2^{m+1}\}$ proves that $U_{u_{2^{m}}} (u_0^*) = u_{2^{m+1}}= U_v(u^*)$ and \begin{equation}\label{eq Delta preserves o-p unitary groups}
\Delta(U_{v}(u^*))= \Delta\left(U_{u_{2^{m}}} (u_0^*) \right) = U_{\Delta(u_{2^{m}})} (\Delta(u_0)^*)=U_{\Delta(v)}(\Delta(u)^*),
\end{equation} for any $u,v$ arbitrary elements in the one-parameter unitary group $\{E_h(t) :{t\in \RR}\}$. By similar arguments applied to $\Delta_0$, or by applying $U_\omega$ at the previous identity, we deduce that \begin{equation}\label{eq Delta preserves o-p unitary groups Delta0} U_{\Delta_0(v)}(\Delta_0(u)^*) = \Delta_0(U_v(u^*)),\end{equation} for every $u,v$ in $\{E_h(t) : {t\in \RR}\}$. Taking $v=\textbf{1} = E_h(0)$ and $u$ arbitrary in \eqref{eq Delta preserves o-p unitary groups Delta0} and having in mind that $\Delta_0(\textbf{1})= \textbf{1}$ we get \begin{equation}\label{eq Delta0 symmetrc} \Delta_0(u)^* = \Delta_0(u^*), \hbox{ for all } u\in \{E_h(t) : {t\in \RR}\}.\end{equation} Furthermore, for any two $u,v\in \{E_h(t) : {t\in \RR}\}$ their adjoined $u^*$, $v^*$ also lie in the set $\{E_h(t) : {t\in \RR}\}$ and thus by \eqref{eq Delta preserves o-p unitary groups Delta0} and \eqref{eq Delta0 symmetrc} we derive $$ U_{\Delta_0(v)}(\Delta_0(u)) 
= U_{\Delta_0(v)}(\Delta_0(u^*)^{*}) = \Delta_0(U_v(u^{**})) = \Delta_0(U_v(u)),$$ for every $u,v$ in $\{E_h(t) : {t\in \RR}\}$, that is,
$$U_{\Delta_0(E_h(t))}(\Delta_0(E_h(s))) = \Delta_0(U_{E_h(t)}(E_h(s))) = \Delta_0(E_h(2 t+s)),$$
 for all $t,s\in \mathbb{R}$.\smallskip

By applying Theorem \ref{t Jordan unitary groups version of Stone's theorem} to $\{\Delta_0(E_h(t)): {t\in \RR}\}$ we deduce the existence (as well as the uniqueness) of a self-adjoint element $y$ in $N$ such that $\Delta_0 (E_h(t)) = e^{i t y}\in \mathcal{U}(N)$ for every $t\in \RR$. \smallskip

We can therefore define a mapping $f:M_{sa}\to N_{sa}$ as the one which maps $h$ into $y$, that is, $f(h)=y$ (where $y$ is the unique element in $N_{sa}$ such that $\Delta_0 (E_h(t)) = e^{i t y}\in \mathcal{U}(N)$ for every $t\in \RR$). Thus, $f$ satisfies \begin{equation}\label{eq f}
\Delta_0(e^{ith})=e^{itf(h)},
\end{equation}for each $t\in\RR$, and each $h\in M_{sa}$. We shall show that $f$ is actually a surjective isometry.\smallskip

Let us first observe that the injectivity of $\Delta_0$ implies that $f$ also is injective. On the other hand, replacing $\Delta_0$ by $\Delta_0^{-1}$ in the previous arguments, we can deduce the existence of an injective mapping $g:N_{sa}\to M_{sa}$ such that \begin{equation}\label{eq g}
\Delta_0^{-1} (e^{ity}) =e^{itg(y)},
\end{equation} for any $y\in N_{sa}$, and any $t\in\RR$. Therefore, by combining the properties of $f$ and $g$, we derive that $y=f(g(y))$ ($y\in N_{sa}$), and hence $f$ is surjective.\smallskip

The bijectivity of $f$ allows us to assure that $\Delta_0$ maps any one-parameter unitary group $\{ e^{ith} \}_{t\in \RR}$ in $\mathcal{U}(M)$ (for any self-adjoint $h$ in $M$) onto the one-parameter unitary group $\{ e^{ity} \}_{t\in\RR}$ in $\mathcal{U}(N)$, with $y\in N_{sa}$.  Consequently, for any $y\in N_{sa}$, and any real number $t$, $$U_{\omega^*}(e^{ity})=U_{\omega^*}(\Delta_0(e^{ith}))=U_{\omega^*}U_\omega(\Delta(e^{ith}))=\Delta(e^{ith}).$$ That proves the first statement, namely, $$\Delta(e^{i M_{sa}})=U_{\omega^*}(e^{i N_{sa}}). $$

Our next goal is to prove that $f$ is an isometry. To this end, given $h,h'\in M_{sa}$ and a real number $t$, let us compute the following limits (with respect to the norm topology) as $t\to 0$:
$$ \frac{e^{ith}-e^{ith'}}{t}=\frac{e^{ith}-\11b{N}}{t}-\frac{e^{ith'}-\11b{N}}{t}\longrightarrow ih-ih', $$
and $$\frac{e^{itf(h)}-e^{itf(h')}}{t}=\frac{e^{itf(h)}-\11b{N}}{t}-\frac{e^{itf(h')}-\11b{N}}{t}\longrightarrow if(h)-if(h').$$
On the other hand, by (\ref{eq f}), we have $$ \|e^{itf(h)}-e^{itf(h')} \|_{_{N}}=\|\Delta_0(e^{ith})-\Delta_0(e^{ith'}) \|_{_{N}}=\|e^{ith}-e^{ith'} \|_{_{N}}. $$ Therefore, by uniqueness of the limits above, $\|f(h)-f(h') \|_{_{N}}=\|h-h' \|_{_{N}}$. The arbitrariness of $h$ and $h'$ in $M_{sa}$ gives the expected conclusion, that is, $f:M_{sa}\to N_{sa}$ is a surjective isometry. Moreover, since $\Delta_0(\11b{M})=\11b{N}$, we deduce that $f(0)=0$, and hence the Mazur--Ulam theorem implies that $f$ is a surjective real linear isometry from $M_{sa}$ onto $N_{sa}$.\smallskip

It is known since the times of Kaplansky that the self-adjoint part of any JB$^*$-algebra is a JB-algebra. Thus, $f:M_{sa}\to N_{sa}$ can be regarded as a linear surjective isometry between JB-algebras. Theorem 1.4 and Corollary 1.11 in \cite{IsRo95} guarantee the existence of a central symmetry $f(\11b{M})$ in $M_{sa}$ and a unital surjective linear isometry $\Phi:M\to N$ such that \begin{equation}\label{eq f description}
f(h)= f(\11b{M})\circ \Phi(h),
\end{equation} for every $h\in M_{sa}$. 
Therefore $\Phi$ is a unital triple isomorphism between unital JB$^*$-algebras, and hence $\Phi$ must be a Jordan $^*$-isomorphism (cf. \cite[Proposition 5.5]{Ka83}).\smallskip

By construction there exists a central projection $p$ in $N$ such that $f(\11b{M})= 2p -\11b{N}= p - (\11b{M}-p)$, where $p$ and $(\11b{N}-p)$ clearly are orthogonal projections in $N$, and for any $n>0$, $$\left( 2p-\11b{N}\right)^n= \left( p-(\11b{N}-p)\right) ^n= p + (-1)^n(\11b{N}-p).$$

Finally, we shall describe $\Delta_0$ in terms of $p$ and $\Phi$. To achieve this goal, we shall employ the equalities obtained in \eqref{eq f}, \eqref{eq f description}, and the definition of the exponential in a Jordan algebra. According to this, given an arbitrary $h\in M_{sa}$, we have
$$\begin{aligned}
\Delta_0(e^{ih})&=e^{if(h)}=e^{if(\11b{M})\circ \Phi(h) }=e^{i(2p-\11b{N})\circ \Phi(h) }= \sum_{n=0}^{\infty}\frac{\left( i  (2p-\11b{N})\circ \Phi(h)\right)^n }{n!}.
\end{aligned}$$
Since $f(\11b{N})$ (and hence $p$) is central, it operator commutes with any element in $N$. Additionally, $\Phi$ is a Jordan $^*$-isomorphism, and we can thus conclude that
$$\begin{aligned}
\Delta_0(e^{ih})&=\sum_{n=0}^{\infty}\frac{ i^n  (2p-\11b{N})^n\circ \Phi(h)^n }{n!} =\sum_{n=0}^{\infty}\frac{ i^n  \left( p + (-1)^n(\11b{N}-p)\right) \circ \Phi(h^n) }{n!} \\
&= \sum_{n=0}^{\infty}\frac{ i^n p \circ \Phi(h^n) }{n!} + \sum_{n=0}^{\infty}\frac{ i^n (-1)^n(\11b{N}-p) \circ \Phi(h^n) }{n!} \\
&=p \circ \sum_{n=0}^{\infty}\frac{ i^n \Phi(h^n) }{n!} + (\11b{N}-p) \circ\sum_{n=0}^{\infty}\frac{ i^n (-1)^n \Phi(h^n) }{n!} \\
&=p \circ \Phi\left( \sum_{n=0}^{\infty}\frac{ i^n h^n }{n!}\right)  + (\11b{N}-p) \circ\Phi\left( \sum_{n=0}^{\infty}\frac{ i^n (-1)^n h^n}{n!}\right) \\
&= p \circ \Phi(e^{ih})  + (\11b{N}-p) \circ\Phi( e^{-ih})= p \circ \Phi(e^{ih})  + (\11b{N}-p) \circ\Phi( e^{ih})^*.
\end{aligned}$$
The arbitrariness of the self-adjoint element $h$ in $M$ gives the following statement
$$ \Delta_0(e^{ih})=U_\omega(\Delta(e^{ih}))=p \circ \Phi(e^{ih})  + (\11b{N}-p) \circ\Phi( e^{ih})^*,\quad (h\in M_{sa}), $$
and consequently,
\begin{align}
\Delta(e^{ih})&=U_{\omega^*}(\Delta(e^{ih}))=U_{\omega^*}\left( p \circ \Phi(e^{ih})  + (\11b{N}-p) \circ\Phi( e^{ih})^*\right) \label{final statement in first theorem} \\
&= U_{\omega^*}\left( P_2(p) \Phi(e^{ih}) \right) + U_{\omega^*} \left( P_2(\11b{N}-p)) \Phi( (e^{ih})^*)\right) \nonumber \\
&=  P_2(U_{\omega^*}( p)) U_{\omega^*}(\Phi(e^{ih}))  + P_2(U_{\omega^*}(\11b{N} -p)) U_{\omega^*}(\Phi( (e^{ih})^*))  \nonumber
\end{align}
because $U_{\omega^*}$ is a triple isomorphism.
\end{proof}

It should be remarked that the idea of employing one-parameter unitary groups was already employed by O. Hatori and L. Moln{\'a}r in \cite{HatMol2014}, where they were motivated by previous results on uniformly continuous group isomorphisms of unitary groups in AW$^*$-factors due to Sakai (see \cite{Sak55}). In our proof this idea is combined with the  Jordan version of Stone's one-parameter theorem developed in Theorem \ref{t Jordan unitary groups version of Stone's theorem}.

\begin{corollary}\label{c new }
Let $\Delta: \mathcal{U} (M)\to \mathcal{U} (N)$ be a surjective isometry, where $M$ and $N$ are two unital JB$^*$-algebras. Then there exist a central projection $p$ in the $\Delta(\11bM)$-isotope $N(\Delta(\11bM))$ and an isometric Jordan $^*$-isomorphism $\Phi:M\to N(\Delta(\11bM))$ such that $$\begin{aligned}\Delta(e^{ih}) &=  p \circ_{\Delta(\11bM)} \Phi(e^{ih}) +  (\11b{N}-p) \circ_{\Delta(\11bM)} \Phi(e^{ih})^{*_{\Delta(\11bM)}} \\
&=  p \circ_{\Delta(\11bM)} \Phi(e^{ih}) +  (\11b{N}-p) \circ_{\Delta(\11bM)} \Phi((e^{ih})^*),
\end{aligned}$$ for all $h\in M_{sa}$.
\end{corollary}

\begin{proof} The desired statement follows from Theorem \ref{t HM for unitary sets of unitary JB*-algebras} by just observing that $\Delta(\11bM)$ is the unit of the $\Delta(\11bM)$-isotope $N(\Delta(\11bM))$, and  $\mathcal{U} (N) = \mathcal{U}(N(\Delta(\11bM)))$ (cf. Lemma \ref{l CR Lemma 4.2.41}$(b)$).
\end{proof}

\begin{remark}\label{r distance to the unit strictly smaller than 2} Let $\Delta: \mathcal{U} (M)\to \mathcal{U} (N)$ be a surjective isometry, where $M$ and $N$ are two unital JB$^*$-algebras. We have shown in the proof of Theorem \ref{t HM for unitary sets of unitary JB*-algebras} (see also Remark \ref{r Jordan starisomorph from a unitary in JBWstar alg}) that the assumption $\|\11b{N}-\Delta(\11b{M})\|<2$ implies the existence of a unitary $\omega$ in $N$ such that $U_{\omega} (\Delta(\11bM)) = \11bN$. So, condition $(1)$ seems stronger, but $(2)$ is all what is needed in the proof of this theorem.
\end{remark}

\begin{remark}\label{r first theorem from triple theory} From the point of view of JB$^*$-triples, the  conclusion of the previous Theorem \ref{t HM for unitary sets of unitary JB*-algebras} can be also stated in the following terms: There exist two orthogonal tripotents $u_1$ and $u_2$ in $M$ and two orthogonal tripotents $\widetilde{u}_1$ and $\widetilde{u}_2$ in $N$, a linear surjective isometry (i.e. triple isomorphism) $\Psi_1 : M_2(u_1)\to N_2(\widetilde{u}_1)$ and a conjugate linear surjective isometry (i.e. triple isomorphism) $\Psi_2 : M_2(u_2)\to N_2(\widetilde{u}_2)$ such that $M= M_2(u_1) \oplus^{\infty} M_2(u_2)$, $N= N_2(\widetilde{u}_1) \oplus^{\infty} N_2(\widetilde{u}_2)$, and the surjective real linear isometry $\Psi =\Psi_1+\Psi_2 : M_2(u_1) \oplus^{\infty} M_2(u_2)\to N_2(\widetilde{u}_1) \oplus^{\infty} N_2(\widetilde{u}_2)$ restricted to $e^{i M_{sa}}$ coincides with $\Delta$.\smallskip

To see this conclusion, we resume the proof of Theorem \ref{t HM for unitary sets of unitary JB*-algebras} from its final paragraph. Set $\widetilde{\Psi}= U_{\omega^*} \Phi$, $\widetilde{u}_1 = U_{\omega^*}( p),$ $\widetilde{u}_2 = U_{\omega^*} (\11b{N}-p),$ ${u}_1=  \widetilde{\Psi}^{-1} (\widetilde{u}_1)$ and ${u}_2=  \widetilde{\Psi}^{-1} (\widetilde{u}_2)$. Since $N = N_2(p) \oplus^{\ell_{\infty}} N_2(\11b{N}-p),$ $U_{\omega^*}$ and $\widetilde{\Psi}$ are triple isomorphisms, we deduce that $N= N_2(\widetilde{u}_1) \oplus^{\infty} N_2(\widetilde{u}_2)$ and $M= M_2(u_1) \oplus^{\infty} M_2(u_2)$. The identity in \eqref{final statement in first theorem} actually proves that
$$\begin{aligned}\Delta ( e^{i h}) &= P_2(U_{\omega^*}( p)) U_{\omega^*}(\Phi(e^{ih}))  + P_2(U_{\omega^*}(\11b{N} -p)) U_{\omega^*}(\Phi( (e^{ih})^*)) \\
&= P_2(\widetilde{u}_1) \widetilde{\Psi} P_2({u}_1) (e^{ih})  + P_2(\widetilde{u}_2) \widetilde{\Psi} P_2({u}_2) ( (e^{ih})^*),
\end{aligned} $$ for all $h\in M_{sa}$. It can be easily checked that the maps $\Psi_1(x) = P_2(\widetilde{u}_1) \widetilde{\Psi} P_2({u}_1) (x)$ and $\Psi_2 (x)  = P_2(\widetilde{u}_2) \widetilde{\Psi} P_2({u}_2) ( x^*)$ ($x\in M$) give the desired statement.
\end{remark}

The next corollary asserts that the Banach spaces underlying two unital JB$^*$-algebras are isometrically isomorphic if and only if the metric spaces determined by the unitary sets of these algebras are isometric.

\begin{corollary}\label{c two unital JB*algebras are isomorphic iff their unitaries are}
Two unital JB$^*$-algebras $M$ and $N$ are Jordan $^*$-isomorphic if and only if there exists a surjective isometry $\Delta: \mathcal{U}(M)\to \mathcal{U}(N)$ satisfying one of the following \begin{enumerate}[$(1)$]\item $\|\11b{N}-\Delta(\11b{M})\|<2$;
\item There exists a unitary $\omega$ in $N$ such that $U_{\omega} (\Delta(\11bM)) = \11bN.$
\end{enumerate}

\noindent Furthermore, the following statements are equivalent for any two unital JB$^*$-algebras $M$ and $N$:
 \begin{enumerate}[$(a)$]
\item $M$ and $N$ are isometrically isomorphic as (complex) Banach spaces;
\item $M$ and $N$ are isometrically isomorphic as real Banach spaces;
\item There exists a surjective isometry $\Delta: \mathcal{U}(M)\to \mathcal{U}(N).$
\end{enumerate}
\end{corollary}

\begin{proof} The first equivalence follows from Theorem \ref{t HM for unitary sets of unitary JB*-algebras}.\smallskip

We deal next with the second statement. The implication $(a)\Rightarrow (b)$ is clear. $(b)\Rightarrow (c)$ Suppose we can find a surjective real linear isometry $\Phi: M\to N$. It follows from \cite[Corollary 3.2]{Da} (or even from \cite[Corollary 3.4]{FerMarPe}) that $\Phi$ is a triple homomorphism, that is, $\Phi$ preserves triple products of the form $\J xyz = (x\circ y^*) \circ z + (z\circ y^*)\circ x -
(x\circ z)\circ y^*$. In particular $\Phi$ maps unitaries in $M$ to unitaries in $N$, and hence $\Phi (\mathcal{U} (M))= \mathcal{U} (N)$. Therefore $\Delta = \Phi|_{\mathcal{U} (M)} : \mathcal{U} (M) \to \mathcal{U} (N)$ is a surjective isometry.\smallskip

The implication $(c)\Rightarrow (a)$ is given by Corollary \ref{c new }.
\end{proof}

As we have seen in Remark \ref{r Jordan starisomorph from a unitary in JBWstar alg}\eqref{eq unitaries are exponentials in JBW*-algebras} for each JBW$^*$-algebra $M$ the set of all unitaries in $M$ is precisely the set $e^{iM_{sa}}$. We are now ready to establish the main result of this paper in which we relax some of the hypotheses in Theorem \ref{t HM for unitary sets of unitary JB*-algebras} at the cost of considering surjective isometries between the unitary sets of two JBW$^*$-algebras.

\begin{theorem}\label{t HM for unitary sets of unitary JBW*-algebras}
Let $\Delta: \mathcal{U} (M)\to \mathcal{U} (N)$ be a surjective isometry, where $M$ and $N$ are two JBW$^*$-algebras. Then there exist a unitary $\omega$ in $N,$ a central projection $p\in N$, and a Jordan $^*$-isomorphism $\Phi:M\to N$ such that $$\begin{aligned}\Delta( u ) &=  U_{\omega^*}\left( p \circ \Phi(u) \right) + U_{\omega^*}\left( (\11b{N}-p) \circ\Phi( u)^*\right)\\
&=  P_2(U_{\omega^*}( p)) U_{\omega^*}(\Phi(u))  + P_2(U_{\omega^*}(\11b{N} -p)) U_{\omega^*}(\Phi( u^*)),
\end{aligned}$$ for all $u\in \mathcal{U}(M)$. Consequently, $\Delta$ admits a (unique) extension to a surjective real linear isometry from $M$ onto $N$.\end{theorem}

\begin{proof} We only need to appeal to the identities $\mathcal{U} (M) =  e^{iM_{sa}}$ and $\mathcal{U} (N) =  e^{iN_{sa}},$ and to the arguments in Remark \ref{r Jordan starisomorph from a unitary in JBWstar alg} to find a unitary $\omega\in N$ such that $U_{\omega} (\Delta(\11bM)) = \11bN$. The rest is a consequence of Theorem \ref{t HM for unitary sets of unitary JB*-algebras}.
\end{proof}

\medskip\medskip

\textbf{Acknowledgements} Authors partially supported by the Spanish Ministry of Science, Innovation and Universities (MICINN) and European Regional Development Fund project no. PGC2018-093332-B-I00, Programa Operativo FEDER 2014-2020 and Consejer{\'i}a de Econom{\'i}a y Conocimiento de la Junta de Andaluc{\'i}a grant number A-FQM-242-UGR18, and Junta de Andaluc\'{\i}a grant FQM375.

\end{document}